\documentclass[12pt]{amsart}

\usepackage{amssymb}
\usepackage{fancyhdr}
\usepackage{bbm}
\usepackage{xcolor}
\usepackage{colonequals}
\usepackage{multirow}
\usepackage{rotating}
\usepackage{fullpage}
\usepackage{hyperref}
\hypersetup{
  colorlinks=true,
  linkcolor=blue,
  citecolor=blue
}
\usepackage{mleftright}

\newtheorem{theorem}{Theorem}
\newtheorem*{theorem*}{Theorem}
\newtheorem{lemma}[theorem]{Lemma}

\newtheorem{cor}[theorem]{Corollary}

\newtheorem{problem}[theorem]{Problem}

\newtheorem{conj}[theorem]{Conjecture}

\theoremstyle{definition}

\newtheorem*{definition*}{Definition}
\newtheorem*{lemma*}{Lemma}
\newtheorem{prop}[theorem]{Proposition}
\usepackage{graphicx}
\usepackage{pstricks, enumerate, pst-node, pst-text, pst-plot}

\numberwithin{equation}{section}
\numberwithin{theorem}{section}

\DeclareMathOperator{\aI}{\textnormal{I}}
\newcommand{\R}{\mathbb{R}}

\renewcommand{\epsilon}{\varepsilon}
\renewcommand{\supset}{\supseteq}

\DeclareMathOperator{\cI}{\mathcal{I}}

\DeclareMathOperator{\cS}{\mathcal{S}}

\DeclareMathOperator{\cZ}{\mathcal{Z}}
\DeclareMathOperator{\Ipla}{\mathcal{I}_{\textnormal{pla}}}
\DeclareMathOperator{\Itre}{\mathcal{I}_{\textnormal{tree}}}
\DeclareMathOperator{\dimm}{\underline{\textnormal{Exp}}}

\title{Independent Sets and Continued Fractions}
\author{Swee Hong Chan, Steven Heilman, Greta Panova}
\date{\today}
\thanks{
Email: sweehong.chan@rutgers.edu, stevenmheilman@gmail.com, gpanova@usc.edu\\
S.H.C. is supported by NSF Grant DMS 2246845.  S.H. is supported by NSF Grant CCF AF 2448108.  G.P. is supported by NSF Grant CCF AF 2302174.  \\
MSC Classification: 05C69, 11A55, 05C05, 05C10\\
Keywords: Independent set, continued fraction, planar graph, tree\\
Department of Mathematics, Rutgers University, Piscataway, NJ, 08854.\\
Department of Mathematics, University of Southern California, Los Angeles, CA 90089}

\begin{document}
\begin{abstract}
Linek's 1989 problem asks whether the numbers of independent sets of trees avoid infinitely many positive integers. We show that the set of natural numbers realized as the number of independent sets of a tree has a lower growth exponent of $0.1966$. We further prove that the set of positive integers representable by connected planar graphs has asymptotic density one. Lastly,  we establish a phase transition: the number of independent sets of graphs with fewer than $d|V|$ edges for any $d<1$ is contained in a set of density zero, whereas, following Shkredov's recent breakthrough on Zaremba's conjecture in continued fraction theory, there exists a constant $D$ such that the number of independent sets of graphs with at most $D|V|$ edges covers all positive integers.
\end{abstract}
\maketitle

\section{Introduction}\label{sec:intro}

An \emph{independent} set in a graph is a set of vertices no two of which are adjacent. For a graph $G$, let $i(G)$ denote the number of independent sets of $G$. This quantity is  also known as the Merrifield--Simmons index or the Fibonacci number of a graph.
Independent sets are among the most basic and most studied objects in graph theory, appearing in extremal graph theory, Ramsey theory, graph coloring, the hard-core model \cite{davies25}, and theoretical computer science \cite{greenhill00,weitz06,sly10}.

The possible values for graphs on $n$ vertices range from linear to exponential: the complete graph $K_{n}$ has $i(K_n)=n+1$, whereas the edgeless graph has $2^n$ independent sets.  Without any restrictions on the graph, the inverse problem is trivial: every positive integer $m$ has at least one graph $G$ with $i(G)=m$.  
In this paper, we study the inverse problem for the map $i$,
which becomes considerably more challenging when the class of allowable graphs is restricted.  Specifically, we ask:

\medskip
Which integers can occur as $i(G)$ when $G$ is drawn from specific classes?
\medskip

The first to systematically study this inverse problem for restricted classes of graphs was Linek~\cite{linek89}, who demonstrated  that $i\colon\{$\textbf{\emph{bipartite}} graphs$\}\to\mathbb{N}=\{1,2,3,\ldots\}$ is surjective.
He further observed that this result does not extend to the more restrictive class of \emph{\textbf{trees}}, as he found specific integers that cannot be realized as the number of independent sets of any tree. This prompted him to pose the following question:

\smallskip

\begin{problem}[\textbf{Linek's Problem}~{\cite{linek89}}]\label{prob:Linek}
Are there infinitely many positive integers that are not the number of independent sets of a tree?

\end{problem}

\smallskip

Virtually no real progress has been made on this problem over the last 35 years. Part of the reason this problem is so difficult is that traditional approaches, such as gadgets and constructions emulating addition and multiplication or those relying on randomized constructions, fail due to Wagner's observation~\cite{wagner09} that almost all random trees have an even number of independent sets (a result further generalized to $F$-matchings by Alon--Haber--Krivelevich~\cite{alon11}).  In fact, \cite{alon11} implies that a large random tree will have a large number of distinct prime factors, implying that a random tree $T$ will rarely satisfy that $i(T)$ is prime, for example, which will not give a negative answer to Problem \ref{prob:Linek}.

On the computational front, Law~\cite{law12} verified that among the integers up to $832040$, only $823$ fail to occur as the number of independent sets of a tree, and the largest missing value is $88013$. In this paper, we extend this evidence substantially: we verify computationally that every integer between  $88014$ and $30$ million is realized as the independent set number of some tree. The missing values are included in the appendix. \footnote{The verification code appears in GitHub \href{https://github.com/sheilman77/independent_sets/}{https://github.com/sheilman77/independent\underline{ }sets/}}
This computational evidence convinces us to make the following conjecture.

\begin{conj}[\textbf{Effective Linek's Problem}]\label{conj:effective_linek}
Every integer greater than $88013$ appears as the number of independent sets of some tree.
\end{conj}

\smallskip

We remark that this conjecture contrasts with a recent conjecture by Han et al.~\cite{han25}, who conjectured that there exist infinitely many positive integers that cannot be realized as the number of independent sets of any tree.

In this paper, we seek to advance the frontier of Linek’s Problem~\ref{prob:Linek} by establishing  theoretical results concerning  three \emph{\textbf{sparse}} graph classes: trees, connected planar graphs, and graphs of bounded average degree.
Our focus on \emph{sparse} graphs is motivated by the observation that the difficulty of Linek’s Problem for trees likely stems from the lack of edges, which are useful in controlling the occurrence of independent sets.
Indeed, the result of Linek for bipartite graphs relied on dense constructions with  $\Omega(|V|^2)$ edges, whereas  Problem \ref{prob:Linek} is for the sparsest connected graphs.

Our main results are as follows:
\begin{itemize}
\item For connected planar graphs, we prove that almost every positive integer occurs as a number of independent sets. 
\item For trees, we show that the number of attainable independent set values up to $N$ grows at least as large as $N^{.1966}$. 
\item For graphs of bounded average degree, we prove a phase transition: below average degree $2$, only a density-zero set of integers can occur as the number of independent sets, whereas above some absolute constant average degree, every positive integer occurs. 
\end{itemize}
Thus, compared with Linek's original bipartite construction, our results show that {\bf graphs with only $O(|V|)$ edges already suffice to realize almost all, and in some regimes all, positive integers}.
We state these three results formally below.

\medskip

\subsection{Trees}
For any $S \subseteq \mathbb{N}$, the \emph{lower growth exponent} of $S$ is 
\[ \dimm(S) \ \colonequals \ \liminf_{N \to \infty}  \frac{\log(|S \cap \{1,\ldots,N\}|)}{\log N}.
\]
We denote by $\Itre \subseteq \mathbb{N}$ the set of natural numbers that can be expressed as the number of independent sets of some tree.
\[  \Itre \ \colonequals \ \{ \, i(T) \, \mid \,  T \text{ is a tree} \, \}.   \]
We also use  the Vinogradov notation
that $f(N) \ll g(N)$ means that
there exists a constant $C > 0$
such that $f(N) \leq C g(N)$ for
all sufficiently large $N$.

\smallskip

\begin{theorem}\label{thm:tre}
The lower growth exponent of $\Itre$ satisfies
\[
\dimm(\Itre)\ge \frac{\log_\tau 2}{4}\approx 0.1966,
\qquad\text{where }\tau\colonequals 1+\sqrt{2}.
\]
Equivalently,
\[ |\Itre \cap \{1,\ldots,N\}| \quad \gg \ N^{\frac{\log_\tau 2}{4}-o(1)}  \ \approx \ N^{0.1966}.  \]
\end{theorem}
\smallskip
We remark that the primary goal of Theorem~\ref{thm:tre} is to establish a strictly positive lower bound on the lower growth exponent, as we do not attempt to optimize the numerical lower bound.
We conjecture that a stronger version of Theorem~\ref{thm:tre} holds, 
which can also be viewed as another weak form of Linek's 
Problem~\ref{prob:Linek}. 

\smallskip
\begin{conj}[\textbf{Density Linek's Problem}]\label{conj:tre}
The set $\Itre$, consisting of integers expressible as the number of independent sets of some tree, has positive lower density, i.e.
\[
\liminf_{N \to \infty} \frac{|\Itre\cap \{1,\ldots,N\}|}{N} \ > \ 0.
\]\end{conj}

In fact, Conjecture \ref{conj:tre} would follow from Hensley's Conjecture in number theory.  See Section \ref{ss:tre} for further discussion on this point.

\medskip

\subsection{Planar graphs}
Denote by $\Ipla \subseteq \mathbb{N}$ the set of integers that can be expressed as the number of independent sets of some connected planar graph,
\[  \Ipla \ \colonequals \ \{ \, i(P) \, \mid \,  P \text{ is a connected planar graph} \, \}.   \]
Let $\mathbb{N}\colonequals\{1,2,3,\ldots\}$.  The \emph{density} of a subset 
$S \subseteq \mathbb{N}$ is the quantity    
 \, $\lim_{N \to \infty} \frac{|S\cap \{1,\ldots,N\}|}{N}$, if the limit exists.  
We say that $S$ \emph{contains almost every integer} if $S$ has  density 1.

\smallskip

\begin{theorem}\label{thm:pla}
\textbf{Almost every} positive integer can be expressed as the number of independent sets of some connected planar graph.
That is, the set $\Ipla$ has density 1.
\end{theorem}

\smallskip

We conjecture that a stronger version of Theorem~\ref{thm:pla} holds, 
which can also be viewed as a weak form of Linek's 
Problem~\ref{prob:Linek}. 
\smallskip

\begin{conj}[\textbf{Planar Linek's Problem}]\label{conj:pla}
    All positive integers can be expressed as the number of independent sets of some connected planar graph.
\end{conj}

\smallskip

This conjecture is further supported by computation, as connected planar graphs realize all $823$ values missing from Law's tree computation.  Note that we consider the empty graph to be a connected planar graph.

\subsection{Graphs with bounded average degree}
 The
\emph{average degree} $d(G)$
of a  nonempty graph $G=(V,E)$ is defined
as $2|E|/|V|$. Let $d>0$.  We
denote by $\cI_{d} \subseteq
\mathbb{N}$ the set
\[ \cI_{d} \ \colonequals \  \{ \,  i(G)
\, \mid \,  G \text{ is a
graph with  }  d(G)\leq d \, \}. \]
Here we do \emph{not} require $G$
to be connected.  We also define $d($empty graph$)\colonequals0$.
\smallskip

\begin{theorem}\label{thm:phase}
There exists a fixed constant $D>0$ such that the following holds
\begin{enumerate}
    \item         If $d \in (0,2)$, then there exist constants $0<c_1(d)\leq c_2(d)<1$  such that 
    \[  N^{c_1(d)}  \quad \ll \quad  |\cI_d \, \cap \,  \{1,\ldots, N\}|  \quad \ll \quad N^{c_2(d)}. \]
\item If $d \geq D$, then every positive integer appears in  $\cI_d$.
\end{enumerate}

\end{theorem}

We conjecture that {\bf the size of
$\cI_d$ exhibits a sharp phase transition and that one can choose $D=2$} in the above theorem.

\begin{conj}\label{conj:phase}
If $d\geq 2$,  then all but finitely many positive integers appear in $\cI_d$. 
\end{conj}
In fact, computations suggest that the only positive integers not appearing are $71$ and $191$, though edge densities $20/9$ and $24/11$ can be achieved for those, respectively.

\smallskip
That is to say, if the average degree is strictly below $2$, then the possible numbers of independent sets  are restricted to a set of integers with a density of zero, whereas for an average degree at least $2$, every positive integer is achievable (except two). Naturally, a negative answer to Linek's Problem~\ref{prob:Linek}  would imply the above conjecture.  Specifically, we have that Conjecture~\ref{conj:effective_linek} (effective Linek's problem) $\Longrightarrow$ Negative answer to Linek's problem~\ref{prob:Linek} $\Longrightarrow$ Conjecture~\ref{conj:tre} (density Linek's Problem) and Conjecture~\ref{conj:phase} (sharp phase transition at $d=2$).  
We also note a key distinction between the assumptions of Theorem~\ref{thm:phase} (for graphs with bounded average degree) and those of Theorems~\ref{thm:tre} (for trees) and \ref{thm:pla} (for connected planar graphs): we do not require the graphs to be connected in the former, whereas we do require them to be connected in the latter.

\smallskip

All three results are proved by applying a connection between independent set counts and continued fractions.
The proof of Theorem~\ref{thm:tre} (trees) draws  inspiration from the work of Alon--Buci\'c--Gishboliner~\cite{ABG}, and is relatively elementary. In contrast, the proof of Theorem~\ref{thm:pla} (planar graphs) requires the more advanced Bourgain--Kontorovich theory~\cite{bourgain14}.
 For Theorem~\ref{thm:phase} (graphs with bounded average degree), the subcritical regime builds on Theorem~\ref{thm:pla} together with standard enumerative arguments, while the supercritical regime uses Shkredov's breakthrough on Zaremba's conjecture~\cite{Shk}.
{{It is worth noting that the phase transition in Theorem~\ref{thm:phase} does not strictly require Shkredov's result; applying Theorem~\ref{thm:pla} instead yields a similar transition, but with the weaker guarantee that $\cI_d$
  is a set of density 1 in part (2)}}. 
 Our continued-fraction viewpoint is in the same spirit as the recent work of Chan--Kontorovich--Pak~\cite{CKP} on Sedl\'{a}\v{c}ek's problem for spanning tree counts.

\smallskip

The paper is organized as follows. In Section~\ref{sec:zar}, we recall various results on continued fraction theory and Zaremba's conjecture. In Sections~\ref{sec:tre}, \ref{sec:pla}, and \ref{sec:phase}, we prove Theorems~\ref{thm:tre}, \ref{thm:pla}, and \ref{thm:phase}, respectively. We conclude with a discussion in Section~\ref{sec:final}.

\section{Zaremba's Conjecture}\label{sec:zar}

In this section, we review results related to Zaremba's conjecture in number theory, which will play a crucial role in the proofs of the main results.

\smallskip

For positive integers $a_1,\ldots, a_\ell \geq 1$, the corresponding \emph{continued fraction} is 
\[
   [a_1,a_2,\dots,a_\ell]
   \;\colonequals\;
   \cfrac1{a_1+\cfrac1{a_2+\cdots+\cfrac1{a_\ell}}}.
\]
It follows from the Euclidean algorithm that every rational number in the interval $(0,1)$ admits a continued fraction expansion.
One of the longstanding open problems in continued fraction theory is the following conjecture of Zaremba~\cite{Zar}.
For a positive integer $A$, let $Q_A$ denote the set of denominators 
of continued fractions whose partial quotients are bounded by $A$, i.e.
\begin{flalign*}
Q_A \colonequals \{ q \in \mathbb{N} \colon \exists\, p\in\mathbb{N}&\text{ with }p < q \text{ such that } 
\gcd(p,q)=1\\
&\text{ and } \tfrac{p}{q} = [a_1,\ldots,a_\ell] \text{ with } 
a_1,\ldots, a_\ell \leq A \}\cup\{1\}.
\end{flalign*}

\smallskip

\begin{conj}[\textbf{Zaremba's Conjecture~\cite{Zar}}]\label{conj:Zar}
For $A=5$, the set $Q_A$ consists of all positive integers.
\end{conj}

\smallskip

Virtually no progress was made on this conjecture until the breakthrough work of Bourgain and Kontorovich~\cite{bourgain14}, who showed that the set $Q_A$ has  density $1$ for $A=50$.  This result was subsequently improved to $A=5$ by Huang~\cite{huang15}.
The latter  will be  a key result used in this paper, so we state it below.

\smallskip

\begin{theorem}[{\cite[Thm~1.6]{huang15}}]\label{thm:Huang}
For $A=5$, the set $Q_A$ has density  equal to $1$.
\end{theorem}

\smallskip

We will also use the following recent breakthrough result of Shkredov.

\smallskip

\begin{theorem}[{\cite[Cor.~1]{Shk}}]\label{thm:Shk}
    There exists an absolute constant $A$ for which the set $Q_A$ contains every prime.
\end{theorem}

\medskip

\section{Trees and the proof of Theorem~\ref{thm:tre}}
\label{sec:tre}

A \emph{marked tree} $(T,v)$ is an ordered pair consisting of a tree $T$ and a root vertex $v$ of $T$.
We denote by $a(T,v)$ the number of independent sets of $T$ containing $v$, and $b(T,v)$ the number of independent sets of $T$ not containing $v$.
We have two operations that generate all marked trees when combined together.
The first operation is the \emph{extension} operation, where given a marked tree $(T,v)$, the new tree $T'$ is obtained by adding a new root vertex $v'$ that is adjacent to $v$, and outputting the new marked tree $(T',v')$. It follows that 
\[  a(T',v') \ \colonequals \  b(T,v), \qquad b(T',v') \ = \  a(T,v) + b(T,v).  \]
The second operation is the \emph{product} operation, where given marked trees $(T_1,v_1)$ and $(T_2,v_2)$, the new tree $T_3$ is obtained by taking the disjoint union of $T_1$ and $T_2$, followed by identifying $v_1$ and $v_2$, and declaring the new vertex to be $v_3$.
It follows that 
\begin{align*}
    a(T_3,v_3) \ = \  a(T_1,v_1)  a(T_2,v_2), \qquad b(T_3,v_3) \ = \ b(T_1,v_1) b(T_2,v_2).
\end{align*}
These operations are standard constructions on trees; and our application here is particularly inspired by \cite{law10}.

    \smallskip

Instead of keeping track of $(a,b)\colonequals(a(T,v),b(T,v))$, it is sometimes more convenient to keep track of two parameters $(g,r)\colonequals(g(T,v),r(T,v))$,
where
\[ g \ \colonequals \  \gcd(a,b), \qquad r(T,v) \ \colonequals \ \frac{a}{b}. \]
Note that $g$ is a positive integer, while $r$ is a rational number in $(0,1]$.
Note that we can recover $(a,b)$ from $(g,r)$ by 
\[ a \ \colonequals \   g \times (\text{numerator of } r(T,v)), \qquad    b \ \colonequals \  g \times (\text{denominator of } r(T,v)). \]
where  $r$ is written in the reduced form.

Let $\cS$ be the set of pairs $(g,r)$ such that there exists a marked tree $(T,v)$ for which $g=g(T,v)$ and $r=r(T,v)$.

\smallskip

\begin{lemma}\label{lem:basic-rec-a}
    Suppose that  $(g,r) \in \cS$ 
    and $(g',r') \in \cS$, where $r'=1/k$ for some positive integer $k$.
    Then the pair $(g'',r'')$ is contained in $ \cS$, where 
    \[  g'' \ \colonequals \  g g'^2 k,   \qquad r'' \ \colonequals \   \frac{1}{k+ r}. \]
\end{lemma}

\smallskip

\begin{proof}
The lemma follows from applying the product operation to $(T,v)$ and $(T',v')$ (corresponding to $(g,r)$ and $(g',r')$, respectively), followed by the extension operation, then followed by product operation with $(T',v')$.
Specifically, let $a,b,a',b'$ be the corresponding numbers of the two trees. The product operation gives a tree with marked vertex $v$ and pair $(aa',bb')$. The extension gives another tree $(T'',v'')$ with $a(T'',v'')=bb'$ and $b(T'',v'')=aa'+bb'$. The final product operation gives a tree $(T_0,w)$ with $a(T_0,w) = bb'a' =bk(g')^2$ and $b(T_0,w) = (aa'+bb')b'= (ag'+kbg')kg'$. The ratio is then
$r(T_0,w) = \frac{bkg'g'}{k(a+kb)g'g'} = \frac{b}{a+kb}=\frac{1}{r+k}$. The gcd is then
$g(T_0,w) = \gcd(bkg'g', kg'g'(a+kb))=k(g')^2\gcd(b,a)=kg(g')^2$. 
\end{proof}

\smallskip

\begin{cor}\label{cor:basic-rec-a}
We have 
    \[ \cS \  \supset  \  \left \{(1,1), \ \Big(1,\frac{1}{2}\Big)\right\},    \]
    and the corresponding trees have a number of vertices equal to $1$ and $2$ respectively.
\end{cor}

\smallskip

\begin{proof}
    This follows by choosing the tree with one and two vertices, respectively.
\end{proof}

\smallskip

Let $\overline{\cS} \subseteq \cS$ be the subset generated by repeatedly
applying the operation to trees described by Lemma~\ref{lem:basic-rec-a}, where at
each step $(g',r')$ is chosen to be one of the two ordered pairs given
in Corollary~\ref{cor:basic-rec-a}.

\smallskip

\begin{lemma}\label{lem:basic-rec-b}
    For every $(g,r) \in \overline{\cS}$,
    \[  g \ \leq \ q,  \]
    where $(p,q)$  are the unique coprime positive integers such that $r=\tfrac{p}{q}$.
\end{lemma}

\smallskip

\begin{proof}
    First note that two ordered pairs in Corollary~\ref{cor:basic-rec-a} clearly satisfy the conclusion of the lemma.
    Now, let  $(g,r)$ be any element satisfying  the conclusion of the lemma, let $(g',r')\in \left \{(1,1), \ \Big(1,\frac{1}{2}\Big)\right\}$ 
    and let $(g'',r'')$ be the resulting element from the operation in Lemma~\ref{lem:basic-rec-a}.
We now show that $(g'',r'')$ also satisfies the conclusion of the lemma.

Let $r=\tfrac{p}{q}$, $r'=\tfrac{1}{k}$,  and $r''=\tfrac{p''}{q''}$.
Suppose that $(g',r')=(1,\tfrac{1}{2})$. From Lemma \ref{lem:basic-rec-a}, $r''=1/(k+r)=1/(k+p/q)=q/(kq+p)$, so $q''=kq+p$.  We then have
\begin{equation}\label{eq:not-optimal}
     g'' \ = \  g(g')^2k \ \leq \  q (g')^2 k  \   = \ q k 
      \leq  \ (kq+p) \ = \ q'',
\end{equation}
    which proves the claim in this case.
    The other case proceeds similarly.

The conclusion of the lemma now follows by induction, as desired.
\end{proof}

We now proceed to prove Theorem~\ref{thm:tre}.
Let $N$ be any sufficiently large positive integer, 
and let  $\tau\colonequals1+\sqrt{2}$.

Set $\ell\colonequals\lfloor  ({\log_{\tau}(N/2)-4})/{2}\rfloor$.
Let $a_1,\ldots, a_\ell$ be arbitrary elements  of $\{1,2\}$.
Note that there are $2^\ell$ many such choices.
We define $P\colonequals P(a_1,\ldots,a_\ell)$ and $Q\colonequals Q(a_1,\ldots,a_\ell)$
to be the unique coprime positive integers such that 
\begin{equation}\label{eq:CD}
    \frac{P}{Q} \ = \ 
    \begin{cases}
        [1,a_1,\ldots,a_\ell] &\text{ if  \ the denominator of $[1,a_1,\ldots,a_\ell]$ is odd};\\
        [1,1,a_1,\ldots,a_\ell] & \text{ if  \ the denominator of $[1,a_1,\ldots,a_\ell]$ is even}
    \end{cases}.
\end{equation}
Note that in either case $Q$ is an odd number.
Also note that $Q$ is maximized when all $a_i$'s are equal to 2, since $P$ and $Q$ are uniquely expressed as polynomials in the $a_i$s with positive coefficients. It then follows from a direct calculation using the matrix product interpretation and that the eigenvalues of the matrix $\begin{pmatrix} 0 & 1\\ 1 & 2\end{pmatrix}$ are $\tau=1+\sqrt{2}$ and $1-\sqrt{2}=-1/\tau$, that
\[ Q \ \leq \ \frac{\tau^{\ell+3}- (-1/\tau)^{\ell+3}}{\tau+1/\tau} \ \leq \ \tau^{\ell+2}. \]

Now note that,
for a given pair of coprime integers $(P, Q)$, at most two tuples $(a_1, \ldots, a_\ell) \in \{1, 2\}^\ell$ satisfy \eqref{eq:CD}.
Hence we conclude that there are at least 
 $\tfrac{2^{\ell}}{2}$
  many distinct pairs of coprime integers $(P,Q)$
such that there exists a tuple $(a_1, \ldots, a_\ell) \in \{1, 2\}^\ell$  satisfying \eqref{eq:CD}.

Now, let $(P,Q)$ be any such pair of integers, and let $(a_1,\ldots,a_\ell)$ be the corresponding tuple.
If there are two of them, then choose the smaller one in lexicographic order.
Let $(T,v)$ be the tree constructed via Lemma \ref{lem:basic-rec-a} and Cor \ref{cor:basic-rec-a} such that 
\[ r(T,v) \ = \  
    \begin{cases}
        [a_1,\ldots,a_\ell] &\text{ if  \ the denominator of $[1,a_1,\ldots,a_\ell]$ is odd};\\
        [1,a_1,\ldots,a_\ell] & \text{ if  \ the denominator of $[1,a_1,\ldots,a_\ell]$ is even}
    \end{cases}.
\]
It follows that $T$ exists by the recursion in Lemma~\ref{lem:basic-rec-a} (setting $r'=1/a_{i}$ in the $i^{th}$ application of the lemma).
By definition of $r(T,v)$ and \eqref{eq:CD}, we have $P/Q = 1/(1+r(T,v))$.  Write $r(T,v)=p/q$ where $\mathrm{gcd}(p,q)=1$.  Then $P/Q = q/(p+q)$ and $\mathrm{gcd}(q,p+q)=\mathrm{gcd}(q,p)=1$, so that
\[ Q \ = \  \text{numerator of } r(T,v) \ + \    \text{denominator of } r(T,v) \ = \ p+q.  \]  
This implies that   
\begin{equation}\label{eq:decomp}
i(T)\ = \ a(T,v)+b(T,v) \ = \  g(T,v) p + g(T,v)q \ = \ g(T,v) Q.
\end{equation}
It follows from the definition of $\overline{\cS}$
that $g(T,v)$ is a power of 2, and 
Lemma \ref{lem:basic-rec-b} implies that
\begin{equation*}
      g(T,v) \ \leq \ q \ \leq  \ Q.
\end{equation*}
Hence we have, by definition of $\ell$,
\[ i(T) \ = \  g(T,v) Q \ \leq  \ Q^{2} \ \leq \ \tau^{2\ell+4} \ \leq \ N/2 \  < \ N.\]
Now let $T'$ be the tree obtained from $T$ by adding a vertex $v'$ that is adjacent to $v$.
By the same argument as before, we have
\begin{equation}\label{eq:decomp-2}
i(T')\ = \ a(T,v)+2b(T,v) \ = \  g(T,v) p + 2g(T,v)q \ = \ g(T,v)P+ g(T,v) Q.
\end{equation}
Hence we have 
\[ i(T') \ \leq  \    2g(T,v) Q \ \leq  \ 2Q^{2} \ \leq \ 2\tau^{2\ell+4} \ \leq \ N.\]
%
Now note that we can recover $(P,Q)$ from $(i(T),i(T'))$, using the following algorithm.
First,  we  can recover the value of $Q$ by taking $Q$ to be the largest odd divisor  of $i(T)$.
Indeed, this is a consequence of \eqref{eq:decomp},
together with the facts that  $Q$ is odd
and $g(T,v)$ is a power of 2.
Then, we can recover the value of $g(T,v)$ by taking $g(T,v)$ to be the largest power of $2$ dividing $i(T)$.
Then, we can recover the value of $P$ by the formula 
\begin{align*}
    P \ = \ \frac{i(T')-i(T)}{g(T,v)}, 
\end{align*}
which follows from \eqref{eq:decomp} and \eqref{eq:decomp-2}.

Combining all the observations above, we conclude that there are at least $m\colonequals2^{\ell}/2$ many trees $T_1,\ldots, T_m$,
such that the vectors $(i(T_j), i(T_j'))$ ($j\in [m]$) are all distinct, and such that 
\[   i(T_j) \leq i(T_j') \leq N.\]
So we have $m$ distinct points in $\{1,\ldots,N\}^2$, and if the distinct $x$-coordinates appearing are $m_1$, and $y$-coordinates are $m_2$, then $m \leq m_1 m_2$, so $m_i \geq \sqrt{m}$ for at least one $i\in \{1,2\}$. But $m_1=|\{i(T_j),j=1,\ldots,m\}|$ and $m_2=|\{i(T'_j), j=1,\ldots,m\}|$ and this implies 
\begin{equation}\label{eq:last}
     | \Itre \cap \{1,\ldots,N\}| \ \geq \  \sqrt{m}.
\end{equation}

Finally, note that by \eqref{eq:last} and by  the definition of $\ell$, 
we have 
\begin{equation*}
     | \Itre \cap \{1,\ldots,N\}| \ \geq \  \sqrt{m}  \ \geq \ \Omega(N^{(\log_\tau 2)/4}),
\end{equation*}

for which the theorem follows.
\qed

\medskip

\section{Planar graphs and proof of Theorem~\ref{thm:pla}}
\label{sec:pla}

Theorem~\ref{thm:pla} is a direct consequence of Theorem~\ref{thm:Huang} and the following proposition.
Here we denote by  $\phi:=\frac{1+\sqrt{5}}{2}$ the \emph{golden ratio}.

\smallskip

\begin{prop}\label{planarprop}
For every $q\in Q_{5}$, there exists a connected planar graph $G$ such that $q$ is equal to the number of independent sets in $G$ and $|V(G)| \leq  5 \log_{\phi} q$.
\end{prop}

\smallskip

\begin{proof}
Let $q \in Q_5$, and let $\tfrac{p}{q}=[a_1,\ldots,a_\ell]$ be the corresponding continued fraction expansion.
Note that $a_1,\ldots,a_\ell \leq 5$ by definition of $Q_{5}$.  For any $a\in\R$, denote 
\begin{equation}\label{eq:ta}
T_{a}\colonequals \begin{pmatrix}0 & 1 \\ 1 & a\end{pmatrix}.
\end{equation}
It is easy to see by induction that$$
\begin{pmatrix}* & *\\ p & q\end{pmatrix}=
T_{a_{\ell}}T_{a_{\ell-1}}\cdots T_{a_{1}},
$$
directly associating the matrix product with
 the continued fraction expansion
$$\frac{p}{q}=[a_1, a_2, \ldots, a_\ell].$$
(Here * denotes entries of the matrix that are ignored.)

A \emph{marked planar graph} $(G, v)$ is a pair of a connected planar graph $G$ 
    and a marked vertex $v$ of $G$.
    Just like with the tree case, let $a(G, v)$ be the number of 
    independent sets of $G$ containing $v$, and $b(G, v)$ be the number 
    of independent sets not containing $v$.
We denote by $\cZ$ the set of pairs $(a,b)$ such that there exists a planar marked graph $(G,v)$  with $a=a(G,v)$ and $b=b(G,v)$.
Note that $(1,1) \in \cZ$, by considering the marked planar graph consisting of a single vertex. 
Moreover, $(1,k) \in \cZ$ for $k=2$ by considering the graph  $P_2$ (path with 2 vertices, one of which is $v$), $k=3$ using $K_3$, $k=4$ using $P_3$ with $v$ the middle vertex and $k=5$ using the graph $P_3 \cup\{v\}$ connecting every vertex of $P_3$ to $v$ ($K_4$ minus an edge).  

We introduce the following gluing operation on marked graphs.
Let $(G, v)$ and $(G', v')$ be two marked planar graphs. Suppose further 
that $G'$ contains a vertex $w'$ that is distinct from $v'$, and both $v'$ and $w'$ are adjacent to each other and every other vertex in $G'$.  
We define $(G'', v'')$ to be the marked graph obtained by taking 
the disjoint union of $G$ and $G'$, identifying $v$ with $v'$, and 
setting the marked vertex to be $v'' = w'$.
See Figure~\ref{fig:planar} for an illustration.
Note that $G''-v''$ is a connected planar graph by construction.

\begin{figure}[ht]
\centering
\includegraphics[width=.6\textwidth]{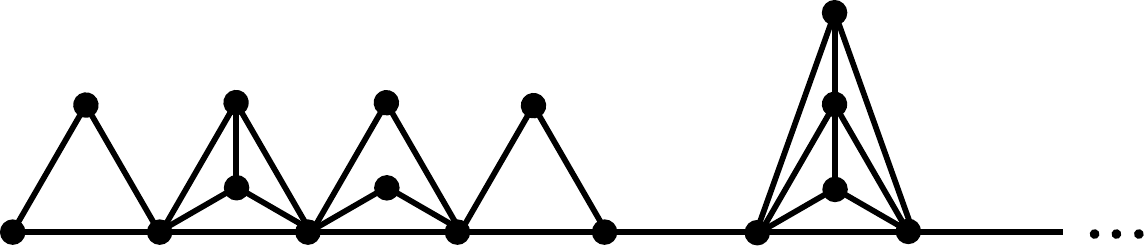}
\caption{A planar graph obtained by consecutive gluing operations in Proposition~\ref{planarprop}. Each ``triangle'' is an instance of a $G'$ graph and the marked vertices are on the bottom line. The corresponding continued fraction is $[2,3,4,2,1,5,\ldots]$.}
\label{fig:planar}
\end{figure}

It follows readily that $(G'',v'')$ is  a marked planar graph satisfying 
\begin{align*}
    a(G'',v'') \ = \ b(G,v), \qquad b(G'',v'') \ =  \ a(G,v) + i(G'-v'-w') b(G,v).
\end{align*}
Therefore it follows that 
\begin{align}\label{eq:tkapp}
    \text{ if } \binom{a}{b} \in \cZ, \ \text{ then } T_{k} \binom{a}{b} \stackrel{\eqref{eq:ta}}{=} \binom{b}{a+kb} \in \cZ,
\end{align}
where $k\colonequals i(G'-v'-w')$.
It therefore remains to construct a family of marked planar graphs $(G', v', w')$ satisfying $i(G' - v' - w') = k$ for each $k \in \{1, \dots, 5\}$.

 \textbf{Case $k \in \{1, 2, 3\}$:} Let $G'$ be the complete 
    graph $K_{k+1}$, and let $v', w' \in V(G')$ be any two distinct 
    vertices. Then $G' - v'- w'$ is isomorphic to $ K_{k-1}$, which 
    implies $i(G' -v'-w') = k$.

 \textbf{Case $k = 4$:} Let $G'$ be the graph obtained by 
    removing one edge from $K_4$, and let $v', w'$ be the two 
    vertices of degree 3. The subgraph $G' - v'-w'$ 
    consists of two isolated vertices, so $i(G'-v'-w') 
    = 4$.

\textbf{Case $k = 5$:}  Let $G'$ be the graph $H$ in Figure~\ref{fig:H}, and let $v'$ and $w'$ be any two vertices of degree 4.
The subgraph $G'-v'-w'$ is isomorphic to the path graph of 3 vertices, so $i(G'-v'-w') 
    = 5$.
    
\begin{figure}[ht]
\centering
\includegraphics[width=.15\textwidth]{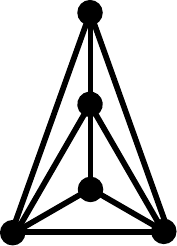}
\caption{The graph $H$ for the case $k=5$ in Proposition~\ref{planarprop}.}
\label{fig:H}
\end{figure}

The final constructed graph $G$ starts with a single vertex (corresponding to $(1,1)\in\cZ$) and then performs $\ell$ gluing operations corresponding to $k$ taking the values $a_{1},\ldots,a_{\ell}$ in \eqref{eq:tkapp}.  Also observe that $|V(G')| \leq 5$ in each of the three cases described above.  Since the Euclidean algorithm terminates in at most $\ell \leq \log_{\phi} q$ steps for any given $q$, our construction yields a graph $G-v$ with $i(G-v)=q$ that satisfies the vertex bound $|V(G-v)| \leq 5 \log_{\phi} q$.
This completes the proof of the proposition.
\end{proof}

\medskip


\section{Graphs with bounded average degree and proof of Theorem~\ref{thm:phase} }
\label{sec:phase}
We split the proof into three parts depending on the average degree $d$ and proving the lower and upper bounds respectively.

\subsection{Proof of the lower bound from Theorem~\ref{thm:phase}}
We will prove the following result.

\smallskip

\begin{lemma}
If $d \in (0,2)$, then there exists a constant $c_1(d)>0$  such that 
    \[  |\cI_d \, \cap \,  \{1,\ldots, N\}|  \quad \gg \quad N^{c_1(d)}. \]    
\end{lemma}

\smallskip

\begin{proof}
Let $N$ be a sufficiently large integer.
Set $k\colonequals \lfloor \log_2 N\rfloor$.
It follows from Proposition~\ref{planarprop} and Theorem \ref{thm:Huang} that, for almost all integers $K \in \{1,\ldots,M\}$ 
there exists a connected planar graph $G\colonequals G(K)\colonequals(V,E)$ such that $i(G)=K$. Let $M=N^{\frac{d}{20(6-d)}}$.
We then have a bound on the number of vertices as
\begin{equation}\label{eq:charlie}
     |V| \ \leq \ 5 \log_{\phi} K  \ < \ 10 \log_2 K  \leq \  \lceil k/2 \rceil \frac{d}{6-d}. 
\end{equation}

Let $G'$ be the graph obtained  by adding $\lceil k/2 \rceil$  isolated vertices into  $G$.
Then 
\[ i(G')  \ = \  2^{\lceil k/2 \rceil} i(G) \ = \  2^{\lceil k/2 \rceil} K. \]
Note that the average degree of $G'$ is given by 
\begin{align*}
    d(G') \ = \  \frac{2|E|}{|V| + \lceil k/2 \rceil}
    \ < \ \frac{6|V|}{|V|+ \lceil k/2 \rceil} \ \leq 
      \ d,
\end{align*}
where the first inequality follows from Euler's formula (i.e., $|E| \le 3|V|-6$ for connected planar graphs),
and the second inequality follows from \eqref{eq:charlie}.  Also note that, by definition of $k,K$ and using $d\leq2$,
\[ i(G') \ = \ 2^{\lceil k/2 \rceil} K \ 
\leq \  (2\sqrt{N}) (N^{\frac{1}{40}}) \ < \ N,\]
for sufficiently large $N$.
Since the choice of $K$ is arbitrary, the lemma now follows by taking $c_1(d) \colonequals  \frac{d}{20(6-d)}$.
    \end{proof}

\medskip

\subsection{Proof of the upper bound in Theorem~\ref{thm:phase} part (a)}
The statement would follow from an elementary number theoretic estimate.

\begin{lemma}\label{lemma:factors}
   Let $N$ be a large integer, let $c\in (0,1)$ and let $k \geq c \log_2(N)$. Then the set of integers $\leq N$ which can be factored as a product of $k$ integers has upper growth exponent $<1$, i.e.
   $$|\{ x \leq N: x=m_1\ldots m_k \text{ for some }m_i\geq 2, m_i\in\mathbb{N}\}| \ll N^\epsilon$$
   for any $\epsilon \in (1-c,1)$. 
\end{lemma}
\begin{proof}
    Denote the above set by $S$ and let $\Omega(x)= \sum_i a_i$ for $x=p_1^{a_1}p_2^{a_2}\cdots$ be the number of prime factors counted with multiplicity. Note that  $x \in S$ if and only if $\Omega(x) \geq k$. 
    We use \cite[Lemma 13]{luca}, which gives 
    $$|S|= \sum_{x \leq N, \, \Omega(x) \geq k} 1 \leq \frac{k}{2^k}N\log N.$$
    We have that $k/2^k \leq c \log_2N/2^{c \log_2 N}=cN^{-c} \log_2N $, so $|S|\leq N^{1-c} (\log N)^2<N^{\epsilon}$ asymptotically for every $\epsilon \in (1-c,1)$.
\end{proof}

We can now prove the relevant portion of Theorem~\ref{thm:phase}.

\begin{lemma}
If $d \in (0,2)$, then there exists a constant $\frac12 + \frac{d}{4}<c_2(d)<1$  such that 
    \[  |\cI_d \, \cap \,  \{1,\ldots, N\}|  \quad \ll \quad N^{c_2(d)}. \]    
\end{lemma}

\begin{proof}
    The idea is that if a graph is sparse then it has many connected components, each contributing a factor in $i(G)$. 
    Let $G$ have $n$ vertices and $k$ connected components $G_1,\ldots,G_k$ and average degree $d$. Then since $G_i$ is connected we have $|E(G_i)| \geq n_i-1$, so
    $$d \geq 2 \frac{ \sum_i |E(G_i)|}{\sum_i n_i} \geq 2 \frac{ n-k}{n} $$
    and hence 
    $k \geq n(1-d/2) =d_1n$ for some $0<d_1<1$.

    We also have that $i(G) = i(G_1)\cdots i(G_k)$ and $i(G_j) \geq 2$ for each component. Let $S'=\{ i(G) \leq N \}$ for graphs $G$ for which $k \geq d_1/2 \log_2N$, and let $S''$ be the set of numbers of independent sets  of graphs with number of connected components $<d_1/2 \log_2 N$. 

    In the first case we apply Lemma~\ref{lemma:factors} since $i(G) = i(G_1)\cdots i(G_k)$ and $i(G_j) \geq 2$ for each component and obtain $|S'| \leq N^\epsilon$ for any $\epsilon \in (1-d_1/2,1)=(1/2+d/4,1)$.

    In the second case we have 
    $d_1n \leq k < d_1/2 \log_2N$, so $n < \frac12 \log_2 N$. For graphs on $n$ vertices we have $i(G) \leq 2^n$ so $i(G) \leq N^{1/2}$ and hence $|S''| \leq N^{1/2}$. 

    Since $\cI_d \cap \{1,\ldots,N\} = S'\cup S''$, we get a bound of $N^{1/2} + N^{\epsilon}\ll N^{\epsilon}$. We can then set $c_2(d) = \epsilon \in \left( \frac12 + \frac{d}{4},1\right)$.  
\end{proof}

\medskip

\subsection{Proof of the completeness  part of Theorem~\ref{thm:phase}}
Here we will show that the numbers of independent sets of graphs with sufficiently large constant average degree can achieve all positive integers.
\smallskip

\begin{lemma}\label{lemma9}
There exists $D>0$ such that 
if $d \geq D$, then $\cI_d$ contains every positive integer.\end{lemma}
\smallskip

\begin{proof}
It suffices to show that $\cI_d$ contains every prime, since every integer can then be obtained by taking disjoint unions of graphs and this does not increase the average degree.    

Let $A$ be the absolute constant from Theorem~\ref{thm:Shk}. This 
implies that for any prime $q$, there exists a positive integer 
$p < q$ such that $p/q = [a_1, \ldots, a_\ell]$ and 
$a_1, \ldots, a_\ell \leq A$. Following the same argument as in the 
proof of Proposition~\ref{planarprop}, it remains to construct a 
family of marked graphs $(G', v', w')$ satisfying 
$i(G' - v' - w') = k$ for each $k \in \{1, \ldots, A\}$. Such a 
family is easily obtained by taking $G'$ to be the complete graph 
$K_{k+1}$. 
Note that attaching an additional component $G' = K_{k+1}$ increases the edge count by $\binom{k+1}{2}$ and the number of vertices by $k$. Consequently, any graph $G$ produced by this construction satisfies the density bound
\[
d(G) \leq \max_{k \in [A]} \frac{2\binom{k+1}{2}}{k} = A+1.
\]
The theorem follows by setting $D := A + 1$.\end{proof}

\section{Final remarks}\label{sec:final}

\subsection{Future direction on trees}\label{ss:tre}
The constant in Theorem~\ref{thm:tre} (main result for trees) could potentially be
improved in a few ways. 
For example, if one were to believe  Hensley's conjecture~\cite{hensley96} (i.e. that 
the set $Q_2$ in Zaremba's Conjecture contains all but finitely many integers), then the argument in our proof of Theorem~\ref{thm:tre} can be used to show that the lower growth exponent of $\Itre$ is at least $1/2$.
However, showing that the lower density of $\Itre$ is positive (which in turn implies that the lower growth exponent is 1) will likely require new arguments.

\medskip

\subsection{Future direction on planar graphs}
There are several ways in which one can improve Theorem~\ref{thm:pla} (main result for planar graphs).
For example, if one were to believe  Zaremba's Conjecture~\ref{conj:Zar},
then the argument in our proof of Theorem~\ref{thm:pla} can already be used to resolve Conjecture~\ref{conj:pla} (the planar Linek's Problem).
However, such an approach is likely unnecessarily strong. Indeed, our current method exploits very little of the combinatorial structure of planar graphs. It is entirely possible that a purely combinatorial proof, utilizing larger families of planar graphs, would suffice without relying on deep results from number theory. However, to resolve this problem combinatorially, one needs a better understanding of the possible pairs $(a,b)$ or generalizations of such quantities. Local operations might not be enough.

\medskip

\subsection{Future direction on graphs with bounded average degree}
We do not attempt to optimize the constants $c_1(d)$ and $c_2(d)$ 
in the first part of Theorem~\ref{thm:phase}, as the primary goal 
of the theorem is simply to bound them away from $0$ and $1$, 
respectively. We believe that for $d \in (0,2)$, the growth exponent 
of $\aI_d$ actually exists, meaning that $c_1(d)$ and $c_2(d)$ can 
be chosen to be arbitrarily close to each other. However, proving 
such a statement appears out of reach with current methods and is 
secondary to the main focus of this work.

A much more interesting question is whether the second part of 
Theorem~\ref{thm:phase} can be improved to show that $D = 2$, 
thereby resolving Conjecture~\ref{conj:phase}.
The current constant $D$ is given by $D = A +1$, where $A$ is the absolute 
constant from Theorem~\ref{thm:Shk}. While $A$ is explicitly 
computable (see \cite[p.~6]{Shk}), this value is likely far from 
optimal.
As is the case 
with Theorem~\ref{thm:pla} for planar graphs, our current approach 
relies on heavy machinery from number theory that is likely 
unnecessary. Finding a more elementary proof of the current result 
could pave the way toward a full resolution of the conjecture.

\medskip

\subsection{Connections between inverse enumeration problems and continued fractions}
Our approach to studying inverse enumeration problems through
the lens of continued fraction theory is inspired by the
recent work of Chan--Kontorovich--Pak~\cite{CKP}. That paper
introduced this strategy to obtain the first known exponential
lower bound for Sedl\'{a}\v{c}ek's problem. Specifically, they
established that
\[ |T(n)| \geq (1.1103)^n, \]
where ${T}(n)$ is the set of integers that can be
realized as the number of spanning trees of a simple
connected graph on $n$ vertices.
The current best lower bound in the literature is $(1.55)^n$, due to Alon, Bucić, and Gishboliner~\cite{ABG}, who also provide a simplified proof.

Another work on inverse enumeration problems that uses similar techniques is the recent work of Chan--Kontorovich--Pak~\cite{CKP25}, which employed Bourgain's results~\cite{Bou12} on the sum of continued fractions to achieve the optimal bound for the \emph{inverse effective resistance problem}.
Finally, we also note the work of Agol and Krushkal \cite{AK} on the exponential lower bound for the \emph{inverse chromatic polynomial problem}, as well as the recent follow-up by Miyazaki, Pohoata, and Zheng \cite{MPZ}, which strengthens the Agol--Krushkal result by unifying their approach with ideas from Alon--Bucić--Gishboliner~\cite{ABG}.

\medskip 

\subsection{Independence polynomials of trees}\label{ss:ind_poly} 
The number of independent sets of trees can be refined via their independence polynomial $$I_G(x) = \sum_{k=0}^{n} i_k(G)x^k,$$
 where $i_k(G)$ is the number of independent sets of size $k$ in $G$. A conjecture of Alavi--Malde--Schwenk--Erd\H{o}s  from 1987 \cite{alavi87} states that the sequence $\{i_k(T)\}_k$ is unimodal when the graph is a tree. Many partial results have provided further evidence towards this conjecture, while also disproving the stronger log-concavity property. The literature on this topic is vast, see some of the most recent developments in~\cite{bencs2025limit,heilman2025independent,galvin2025trees,ramos2025ai} and references therein. While we study the inverse problem of finding which integers are equal to $I_T(1)$, it would be a much more challenging question to characterize the independence sequences $\{i_k(T)\}_k$.

\medskip

\subsection{Caterpillars} A special class of trees, also of great relevance to \S\ref{ss:ind_poly}, are the caterpillar trees which consist of a spine of vertices $v_1, v_2,\ldots,v_k$ with $(v_i,v_{i+1})\in E$ and ``cherries'' of leaves connected to each vertex, so vertex $v_i$ has $a_i$ many leaves connected to it. It is not very hard to see that the trees constructed in Section \ref{sec:tre} are actually from that class. Computational results show that the numbers of independent sets of caterpillars already cover close to $1/2$ of the integers in large enough intervals and we conjecture that this is indeed the case. 

\subsection{Vertex minimality} A related problem concerns optimizing with respect to the number of vertices. In particular, what is the smallest number of vertices $v=v(N)$ needed to construct a graph with number of independent sets $N$.  For instance, \cite[Ex.~5.8]{CPcoincidences} demonstrates a construction where $\log_2 N \leq v(N) \leq 2 \log_2 (N+1)$.
If we assume that trees achieve all but finitely many numbers (Conjecture~\ref{conj:effective_linek}), it is then easy to see $N \geq F_{v(N)+2}$, so $v(N) \leq \log_\phi N$ (i.e. a multiplicative constant improvement for the upper bound). 
We can then ask: for a given $N$, what is the asymptotic behavior of the ratio $\frac{\log v(N)}{\log N}$?


\section*{Acknowledgement}
We thank Lior Gishboliner, Alex Kontorovich, and Cosmin Pohoata for stimulating discussions on these problems and Noga Alon and Igor Rivin for interesting questions and directions.
The first author would like to thank the National University of Singapore for their warm hospitality during his sabbatical in Spring 2026.

\section{Appendix: Forbidden Integers}\label{sec:app}

Here is a list of the 823 known integers that cannot be equal to the number of independent sets of a tree.

4
6
7
10
11
12
15
16
18
19
20
25
27
28
29
30
31
32
39
42
45
46
47
48
49
51
52
53
54
56
63
67
71
72
73
74
75
78
79
81
82
85
86
87
88
90
91
103
111
115
117
119
123
125
127
130
131
132
135
136
137
138
139
140
141
142
143
146
147
151
155
177
179
185
191
193
201
203
205
207
208
210
211
213
214
215
219
220
221
223
227
228
229
230
231
235
237
238
239
240
244
247
251
265
271
291
295
299
301
303
315
322
325
329
331
333
337
339
341
343
345
347
350
355
357
358
359
361
362
363
369
373
374
379
385
387
389
391
399
400
407
413
427
437
455
463
467
475
477
481
483
487
497
507
509
515
519
521
529
535
539
541
543
551
555
559
561
563
567
573
575
577
579
583
585
591
593
594
605
607
609
611
621
622
623
625
626
627
628
629
633
639
641
645
646
647
656
677
683
699
703
709
719
721
759
787
797
805
813
827
837
839
847
851
855
859
869
871
875
883
887
891
897
901
907
911
913
915
921
923
925
927
933
935
937
939
943
949
953
955
961
963
967
969
971
973
975
977
979
989
991
993
995
999
1001
1003
1007
1009
1013
1015
1017
1018
1019
1023
1031
1041
1043
1045
1047
1063
1069
1081
1083
1095
1103
1111
1117
1123
1129
1159
1171
1179
1189
1235
1247
1281
1287
1333
1339
1363
1383
1385
1389
1397
1439
1441
1459
1461
1463
1469
1477
1493
1513
1517
1519
1523
1525
1526
1527
1531
1539
1541
1553
1567
1569
1571
1573
1581
1595
1598
1603
1605
1615
1623
1625
1631
1635
1639
1642
1643
1647
1651
1655
1657
1667
1671
1679
1685
1687
1691
1703
1711
1723
1729
1747
1751
1759
1763
1767
1771
1791
1793
1817
1823
1829
1871
1923
1931
1937
2081
2111
2119
2171
2179
2195
2209
2237
2239
2249
2269
2287
2291
2299
2353
2355
2359
2369
2377
2379
2381
2399
2403
2411
2427
2451
2457
2471
2481
2485
2491
2497
2499
2513
2515
2517
2519
2539
2549
2555
2567
2569
2571
2577
2593
2597
2599
2605
2611
2619
2621
2625
2633
2639
2643
2653
2657
2667
2677
2685
2713
2715
2721
2731
2737
2743
2755
2758
2761
2771
2811
2815
2821
2831
2835
2839
2841
2861
2863
2921
2933
2941
2951
2975
2989
3061
3113
3119
3137
3259
3287
3371
3389
3403
3409
3439
3515
3535
3601
3613
3619
3635
3647
3679
3763
3775
3791
3831
3847
3867
3871
3895
3913
3923
3927
3931
3935
3955
3967
3979
3983
3985
3991
4013
4039
4043
4049
4063
4067
4069
4081
4095
4103
4107
4111
4113
4119
4139
4151
4165
4175
4177
4189
4195
4199
4203
4207
4209
4229
4233
4259
4263
4287
4303
4305
4331
4337
4367
4375
4379
4411
4417
4431
4433
4435
4451
4465
4475
4487
4519
4521
4522
4529
4547
4555
4561
4565
4579
4603
4607
4619
4657
4667
4675
4691
4727
4739
4741
4787
4795
4803
4811
4867
4907
4981
4997
5011
5033
5041
5099
5173
5227
5417
5423
5581
5591
5603
5611
5653
5719
5725
5821
5885
5939
5993
5997
5999
6167
6171
6223
6301
6313
6323
6451
6477
6491
6499
6539
6563
6575
6577
6607
6617
6627
6653
6727
6739
6751
6755
6757
6803
6827
6901
6929
6941
6947
6973
6977
6983
7049
7055
7059
7067
7117
7169
7175
7193
7203
7211
7265
7279
7343
7357
7369
7371
7379
7385
7401
7419
7431
7441
7451
7457
7459
7463
7489
7517
7559
7619
7649
7679
7733
7769
7775
7811
7817
7819
7847
7927
7987
8041
8047
8057
8059
8107
8117
8167
8261
8393
8399
8461
8467
8831
9023
9089
9131
9181
9463
9581
9611
9647
9691
9767
9811
9863
9883
10037
10079
10093
10127
10293
10399
10581
10733
10747
10751
10775
10789
10999
11027
11029
11167
11183
11257
11353
11413
11425
11443
11503
11543
11603
11639
11669
11671
11759
11769
11793
11827
11969
11989
12019
12079
12089
12121
12131
12323
12329
12517
12547
12565
12583
12639
12749
12931
12953
12979
12989
13015
13051
13099
13111
13139
13199
13291
13363
13501
13513
13607
13643
13807
13879
14483
14993
15329
16643
16693
16751
16879
17131
17183
17233
17427
17471
17507
17571
17759
17837
18221
18275
18383
18389
18667
18679
18893
18899
18943
19139
19207
19265
19333
19477
19481
19565
19765
19939
20009
20373
20387
20447
20537
20543
20659
21223
21419
21527
21989
22183
22453
24101
24239
24433
25499
26167
26219
26719
27167
27527
29371
29537
29639
30319
30471
30821
31493
31997
32699
32767
32857
32873
33031
33073
34139
34483
34789
34909
36619
37027
43039
45983
47401
48917
49253
55487
88013


\newcommand{\etalchar}[1]{$^{#1}$}
\providecommand{\bysame}{\leavevmode\hbox to3em{\hrulefill}\thinspace}
\providecommand{\MR}{\relax\ifhmode\unskip\space\fi MR }
\providecommand{\MRhref}[2]{%
  \href{http://www.ams.org/mathscinet-getitem?mr=#1}{#2}
}
\providecommand{\href}[2]{#2}

\end{document}